\documentclass[12pt,a4paper, twoside,reqno]{amsart}
\usepackage{amscd}
\usepackage{amsmath,amssymb,amsthm,mathrsfs}
\usepackage[colorlinks=true,linkcolor=blue,citecolor=blue]{hyperref}
\usepackage[margin=3cm]{geometry}
\usepackage{tikz}
\usetikzlibrary{calc}

\usepackage{parskip}
\usepackage{enumerate}
\newtheorem{thm}{Theorem}[section]
\newtheorem*{thm*}{Theorem}
\newtheorem{prop}[thm]{Proposition}
\newtheorem{lem}[thm]{Lemma}
\newtheorem{cor}[thm]{Corollary}

\theoremstyle{definition}
\newtheorem{definition}[thm]{Definition}

\theoremstyle{remark}
\newtheorem{remark}[thm]{Remark}
\numberwithin{equation}{section}

\title{Area Statistics for Large Oscillating Tableaux}

\author{David Keating}
\address{D.K.: Department of Mathematics, University of California, Berkeley,
CA 94720, USA}
\email{dkeating@berkeley.edu}

\begin{document}
\begin{abstract}
In this note we show that the area of the partitions making up an oscillating tableaux is described by a random walk on the first quadrant of $\mathbb{Z}^2$ with certain position dependent weights.  We are able to recursively calculate the moments of the walk. As the length of the oscillating tableaux becomes large we show that this random walk converges to a Gaussian stochastic process.
\end{abstract}
\maketitle
\tableofcontents

Oscillating tableaux can be viewed as random walks on Young's lattice of fixed length, beginning at the empty partition and ending at some partition $\lambda$. Here we show that the area of the partitions making up a tableaux can themselves be view as a random walk in the first quadrant of $\mathbb{Z}^2$ with certain position dependent weights we will describe. In section \ref{sec:1}, we review some background in order to motivate what will follow. In section \ref{sec:2}, we study the random walk on $\mathbb{Z}^2$. We show that it is possible to recursively solve for the moments of the walk. In sections \ref{sec:3} and \ref{sec:4}, we look at the limiting behavior of the walk as the length becomes large. 

\section{Preliminaries} \label{sec:1}

\begin{definition}
An oscillating tableaux of length $N$ and shape $\lambda$ is a sequence of partitions
\[
\{\lambda^{(0)} = \emptyset,\lambda^{(1)},\ldots,\lambda^{(N-1)},\lambda^{(N)}=\lambda\}
\]
such that either $|\lambda^{(i+1)}|/|\lambda^{(i)}| = 1$ or  $|\lambda^{(i)}|/|\lambda^{(i+1)}| = 1$ for all $i<N$. Viewing the partitions as their Young diagrams, this means each subsequent partition differs from the previous partition by adding or removing a single corner box. See figure \ref{fig:OTex}.
\end{definition}

\begin{figure}[h]
\begin{center}
\[
\left\{ \lambda^{(0)}=\emptyset, 
\lambda^{(1)}=
\begin{tikzpicture}[baseline=-7]
\draw (0,0)--(0.2,0)--(0.2,-0.2)--(0,-0.2)--(0,0);
\end{tikzpicture},
\lambda^{(2)}=
\begin{tikzpicture}[baseline=-7]
\draw (0,0)--(0.2,0)--(0.2,-0.4)--(0,-0.4)--(0,0); \draw (0,-0.2)--(0.2,-0.2);
\end{tikzpicture},
\lambda^{(3)}=
\begin{tikzpicture}[baseline=-7]
\draw (0,0)--(0.4,0)--(0.4,-0.2)--(0.2,-0.2)--(0.2,-0.4)--(0,-0.4)--(0,0); \draw (0.2,0)--(0.2,-0.2);  \draw (0,-0.2)--(0.2,-0.2);
\end{tikzpicture},
\lambda^{(4)}=
\begin{tikzpicture}[baseline=-7]
\draw (0,0)--(0.4,0)--(0.4,-0.2)--(0,-0.2)--(0,0); \draw (0.2,0)--(0.2,-0.2);
\end{tikzpicture},
\lambda^{(5)}=
\begin{tikzpicture}[baseline=-7]
\draw (0,0)--(0.4,0)--(0.4,-0.2)--(0.2,-0.2)--(0.2,-0.4)--(0,-0.4)--(0,0); \draw (0.2,0)--(0.2,-0.2);  \draw (0,-0.2)--(0.2,-0.2);
\end{tikzpicture}
\right\}
\]
\end{center}
\caption{An example of an oscillating tableaux of shape $\lambda = (2,1)$ and length $N=5$.}\label{fig:OTex}
\end{figure}

We denote the set of all oscillating tableaux of length $N$ and shape $\lambda$ by $\mathcal{OT}(\lambda,N)$. It is a well-known fact that the number of such tableaux for fixed shape and length has a very simple formula. 
\begin{thm}\label{thm:enum}
Let $\lambda$ be a partition with $|\lambda| = k$. Then for all $n\in  \mathbb{N}$ let $N=k+2n$, we have
\[
\# \mathcal{OT}(\lambda,N) = \binom{N}{k} (N-k-1)!! f^\lambda
\]
 where $f^\lambda$ is the number of standard Young tableaux of shape $\lambda$. Further, $\# \mathcal{OT}(\lambda,l) = 0$ if $l \ne k+2n$ for some $n\in\mathbb{N}$.
\end{thm}
There are many proofs of this theorem, for examples see \cite{R,St1,St2,S}. Here we provide a proof in the style of \cite{GNW1,GNW2} in order to motivate what will follow.
\begin{proof}
As the second part of the theorem follows from a simple parity argument, we only consider the case when the number of oscillating tableaux is nonzero.

Let $\lambda$ be a partition with $|\lambda|=k$, and let $N = k+2n$ for some $n\in\mathbb{N}$. Let $\tilde f^\lambda_N =  \binom{N}{k} (N-k-1)!! f^\lambda$. 

Clearly the theorem is true when $N=0$. Fix a partition $\lambda$. Consider an oscillating tableaux in $\mathcal{OT}(\lambda,N)$. Restricting to the first $N$ partitions $\{ \lambda^{(0)} = \emptyset,\ldots,\lambda^{(N-1)}=\mu \}$ gives an oscillating tableaux in $\mathcal{OT}(\mu,N-1)$. In fact, every oscillating tableaux in $\mathcal{OT}(\lambda,N)$ can be constructed from an oscillating tableaux in $\mathcal{OT}(\mu,N-1)$ for some $\mu$ which differs from $\lambda$ by the addition or subtraction of a single box. Thus the theorem follows by an induction argument if we can show that
\[
\tilde f^\lambda_N = \sum_{\mu \subset \lambda} \tilde f^\mu_{N-1} +  \sum_{\mu \supset \lambda} \tilde f^\mu_{N-1}.
\]
Equivalently, we will verify that
\[
 \sum_{\mu \subset \lambda} \frac{\tilde f^\mu_{N-1}}{\tilde f^\lambda_N} +  \sum_{\mu \supset \lambda} \frac{\tilde f^\mu_{N-1}}{\tilde f^\lambda_N} = 1.
\]
From the definition of $\tilde f^\lambda_N$, the LHS becomes
\begin{align*}
& \sum_{\mu \subset \lambda} \frac{\binom{N-1}{k-1} (N-k-1)!!}{\binom{N}{k} (N-k-1)!!} \frac{f^\mu}{f^\lambda} + \sum_{\mu \supset \lambda} \frac{\binom{N-1}{k+1} (N-k-3)!!}{\binom{N}{k} (N-k-1)!!} \frac{ f^\mu}{ f^\lambda} \\
= & \frac{k}{N} \sum_{\mu \subset \lambda} \frac{ f^\mu}{ f^\lambda} + \frac{N-k}{N} \sum_{\mu \supset \lambda} \frac{1}{k+1} \frac{ f^\mu}{ f^\lambda} \\
= & \frac{k}{N} + \frac{N-k}{N} \\
= & 1
\end{align*}
where the second-to-last equality follows from the well-known identities $f^\lambda = \sum_{\mu \subset \lambda} f^{\mu}$ and $f^\lambda =\frac{1}{k+1} \sum_{\mu \supset \lambda} f^{\mu}$.
\end{proof}

Note that this is essentially the content of Lemma 2.2 in \cite{HX}.

\begin{remark}
This provides a nice algorithm for uniformly sampling $\mathcal{OT}(\lambda,N)$ that interpolates the two hook-walk algorithms developed in \cite{GNW1,GNW2}. 
\begin{enumerate}
\item Fix $N$ and $\lambda$ such that $|\lambda|=k$ and $N=k+2n$ for some $n\in\mathbb{N}$.
\item Set $\lambda^{(N)} = \lambda$. Set $X=0$. Set $Y=k$.
\item While $X < N$:
\begin{enumerate}
\item With probability $\frac{Y}{N-X}$, remove a corner from $\lambda^{(N-X)}$ using the hook-walk algorithm from \cite{GNW1}. Set the new partition to $\lambda^{(N-X-1)}$. Take $X\to X+1$, $Y\to Y-1$.
\item Otherwise, add a corner to $\lambda^{(N-X)}$ using the complementary hook-walk algorithm from \cite{GNW2}. Set the new partition to $\lambda^{(N-X-1)}$. Take $X\to X+1$, $Y\to Y+1$.
\end{enumerate}
\item Return the oscillating tableaux $\{\lambda^{(0)} = \emptyset,\lambda^{(1)},\ldots,\lambda^{(N-1)},\lambda^{(N)}=\lambda\}$.
\end{enumerate}
The fact that this algorithm samples uniformly follows directly from the induction argument used to prove (\ref{thm:enum}).
\end{remark}

\begin{remark}\label{rmk:AreaRW}
Note that the choice of the whether or not to add or remove a box from the partition at step $X$ is dependent only on $X$ and the area of the partition, not on any other details about the shape of the partition. This means that the area of the partitions in an oscillating tableaux can themselves be described by a random walk, independent of the details of the partitions' shape. We will focus on this random walk for the remainder of the note.
\end{remark}

\section{Area Random Walk} \label{sec:2}
In light of the observation in (\ref{rmk:AreaRW}) at the end of the previous section, consider the random walk on $\mathbb{Z}^2$ with elementary steps $(1,\pm1)$. Denote an instance of the random walk by $H:\mathbb{Z} \to \mathbb{Z}$, where $Y=H(X)$ is the height of the walk at step $X$. See figure \ref{fig:AreaEx}.

 Fix the length $N$ and starting point $H(0)=Y_0$, $0\le Y_0 \le N$. Then the position-dependent probabilities of the elementary steps are given by
\begin{equation}
\begin{aligned}
& \mathbb{P}[H(X+1)=Y+1 | H(X) = Y] = 1- \frac{Y}{N-X} = \frac{N-X-Y}{N-X} \\
& \mathbb{P}[H(X+1)=Y-1 | H(X) = Y] = \frac{Y}{N-X}.
\end{aligned}
\end{equation}
From the definition it is obvious this random walk is Markov. It is easy to check that $\mathbb{P}[H(N)=0]=1$. 

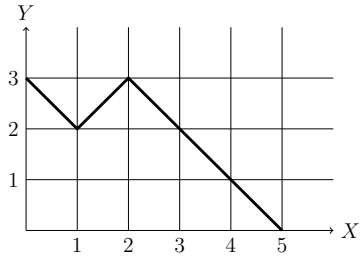
\begin{figure}
\begin{center}
\resizebox{5cm}{!}{
\begin{tikzpicture}
\draw[ultra thick] (0,3)--(1,2)--(2,3)--(3,2)--(4,1)--(5,0);
\draw[->] (0,0)--(0,4); \draw[->] (0,0)--(6,0);
\draw[thin] (1,0)--(1,4); \draw[thin] (2,0)--(2,4); \draw[thin] (3,0)--(3,4); \draw[thin] (4,0)--(4,4); \draw[thin] (5,0)--(5,4);
\draw[thin] (0,1)--(6,1); \draw[thin] (0,2)--(6,2); \draw[thin] (0,3)--(6,3); 
\node[left] at (0,1) {1}; \node[left] at (0,2) {2}; \node[left] at (0,3) {3};
\node[below] at (1,0) {1}; \node[below] at (2,0) {2}; \node[below] at (3,0) {3}; \node[below] at (4,0) {4}; \node[below] at (5,0) {5};
\node[right] at (6,0) {$X$}; \node[above] at (0,4) {$Y$};
\end{tikzpicture}
}
\end{center}
\caption{The corresponding instance of the area walk for the oscillating tableau given in fig. \ref{fig:OTex}.}\label{fig:AreaEx}
\end{figure}

Due to the Markov property of the random walk, it is easily to calculate the moments recursively as in the following proposition.
\begin{prop}\label{prop:recMoments}
Fix $N$ and $Y_0$. The $n^{th}$ moment of the random walk is given recursively by
\begin{equation}
\mathbb{E}[H(X+1)^n] = 1 + \sum_{k=1}^n \left( \binom{n}{k} - (1+(-1)^{n-k})\binom{n}{k-1} \frac{1}{N-X}\right)\mathbb{E}[H(X)^k]
\end{equation}
for $X\in\{0,\ldots,N-1\}$, with $\mathbb{E}[H(0)^n] = Y_0^n$.
\end{prop} 
\begin{proof}
From the law of total expectation we have
\[
\mathbb{E}[H(X+1)^n] = \mathbb{E}_Y[ \mathbb{E}[H(X+1)^n | H(X)=Y]]
\]
where $\mathbb{E}_Y$ denotes the expected value with respect to the random variable $Y$. Expanding the conditional expectation on the RHS using the Markov property gives
\begin{align*}
\mathbb{E}[H(X+1)^n | H(X)=Y] & = (Y+1)^n \left(1- \frac{Y}{N-X}\right)  + (Y-1)^n  \frac{Y}{N-X} \\
& = \sum_{k=0}^n \binom{n}{k} Y^k\left(1- \frac{Y}{N-X}\right) + \sum_{k=0}^n \binom{n}{k} (-1)^{n-k}\frac{Y^{k+1}}{N-X} \\
& = \sum_{k=0}^n \binom{n}{k} Y^k - \sum_{k=0}^n \binom{n}{k} (1- (-1)^{n-k})\frac{Y^{k+1}}{N-X} \\
& = 1+ \sum_{k=1}^n \left(\binom{n}{k} -(1+ (-1)^{n-k}) \binom{n}{k-1} \frac{1}{N-X}\right) Y^k.
\end{align*}
Note that the terms depending on $Y^{n+1}$ cancel. Taking the expected value with respect to $Y$ gives the result.
\end{proof}

\begin{cor}
\begin{equation}
\begin{aligned}
 \mathbb{E}[H(X)^n] = (N-X)P_n(X,Y_0) & \text{ for } X>0 
\end{aligned}
\end{equation}
where $P_n(x,y)$ is a polynomial in two variables. 
\end{cor}
\begin{proof}
We'll prove this with induction on $n$ and the length of the walk $N$. Suppose it holds for walks of length $N-1$. We have that
\[
\begin{aligned}
\mathbb{E}[H(X)^n] & = \mathbb{E}_{Y_1}[ \mathbb{E}[H(X)^n|H(1)=Y_1]]\\
& =\mathbb{E}[(N-X)P_n(X-1,H(1))] \\
& = \mathbb{E}[(1-\frac{Y_0}{N})P_n(X-1,Y_0+1)+\frac{Y_0}{N}P_n(X-1,Y_0-1)].
\end{aligned}
\]
The second equality follows since conditioning on $H(X_1)$ results in a new walk with $N\to N-1$, $X\to X-1$, and initial height $Y_1$.  Since the last line is a polynomial in $X$ and $Y_0$, so is $\mathbb{E}[H(X)^n]$. Since $\mathbb{E}[H(N)^n]=0$, the polynomial must have a factor of $(N-X)$ which finishes the proof. 
\end{proof}

In particular, we can solve this recurrence for $n=1,2$ (and $N\ge2n$) to get the rather messy formulas
{\tiny
\begin{equation}\label{eq:smallMoments}
\begin{aligned}
\mathbb{E}[H(X)] = & \frac{X(N-X)}{N-1} + Y_0 \frac{(N-X)(N-X-1)}{N(N-1)} \\
\mathbb{E}[H(X)^2] = & \frac{X(N-X)(N X - X^2 -2)}{(N-1)(N-3)} + Y_0 \frac{2X(N-X)(N-X-1)((N-1)^2-(N-1)X)}{N(N-1)(N-2)(N-3)}\\
& + Y_0^2 \frac{(N-X)(N-X-1)(N-X-2)(N-X-3)}{N(N-1)(N-2)(N-3)} \\
Var[H(X)] = & \mathbb{E}[H(X)^2]-\mathbb{E}[H(X)]^2 \\
= & \frac{2X(X-1)(N-X)(N-X-1)}{(N-1)^2(N-3)} + Y_0 \frac{2X(2N^2-5N+1-(3N-5)X)(N-X)(N-X-1)}{N(N-1)^2(N-2)(N-3)} \\
& - Y_0^2 \frac{2X(2N ^2-6N+3-(2N-3)X)(N-X)(N-X-1)}{N^2(N-1)^2(N-2)^2(N-3)}.
\end{aligned}
\end{equation}
}
\begin{figure} [!htb]
\begin{minipage}[c]{0.45\linewidth}
\includegraphics[width=\linewidth]{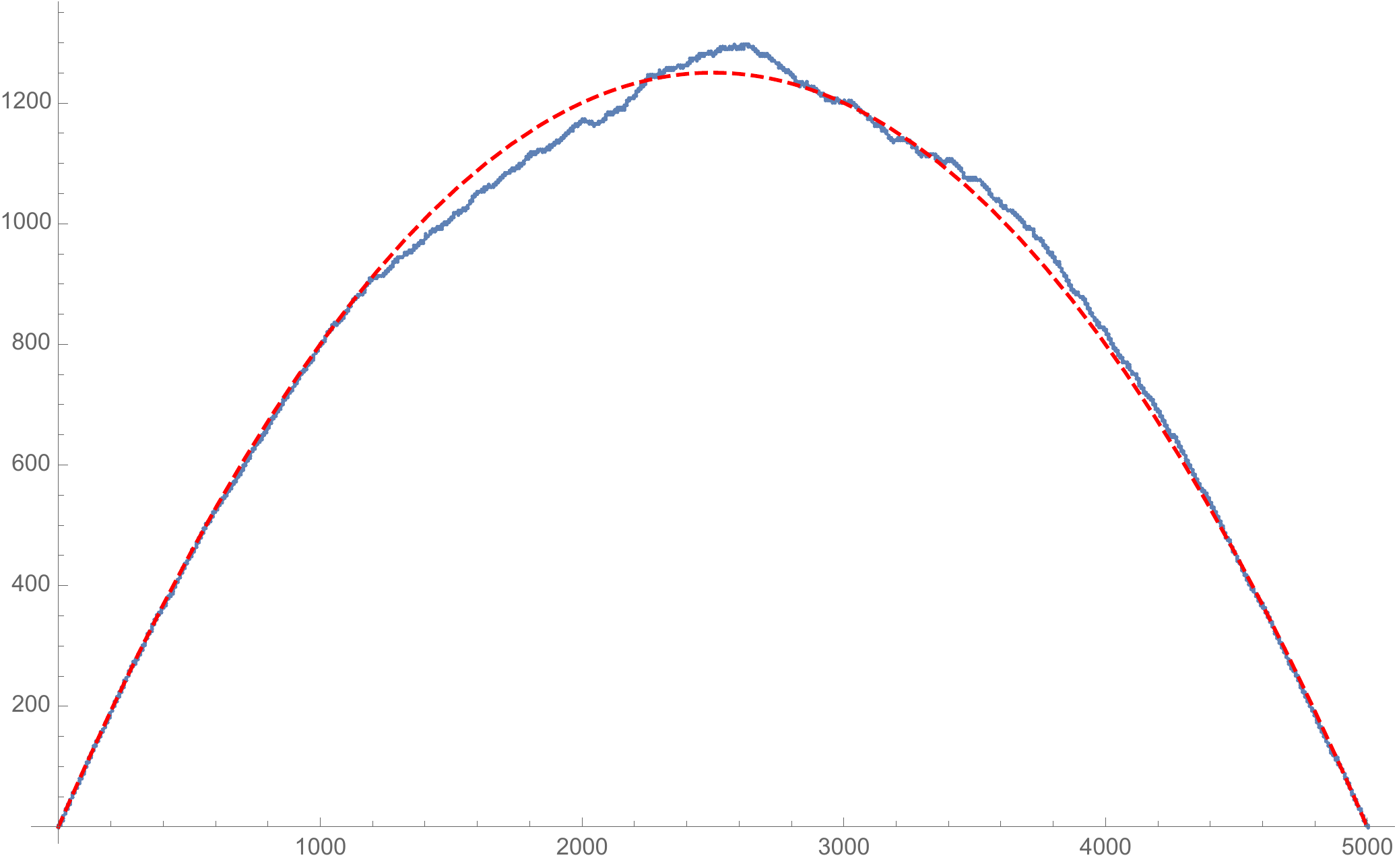}

\end{minipage}
\hfill
\begin{minipage}[c]{0.45\linewidth}
\includegraphics[width=\linewidth]{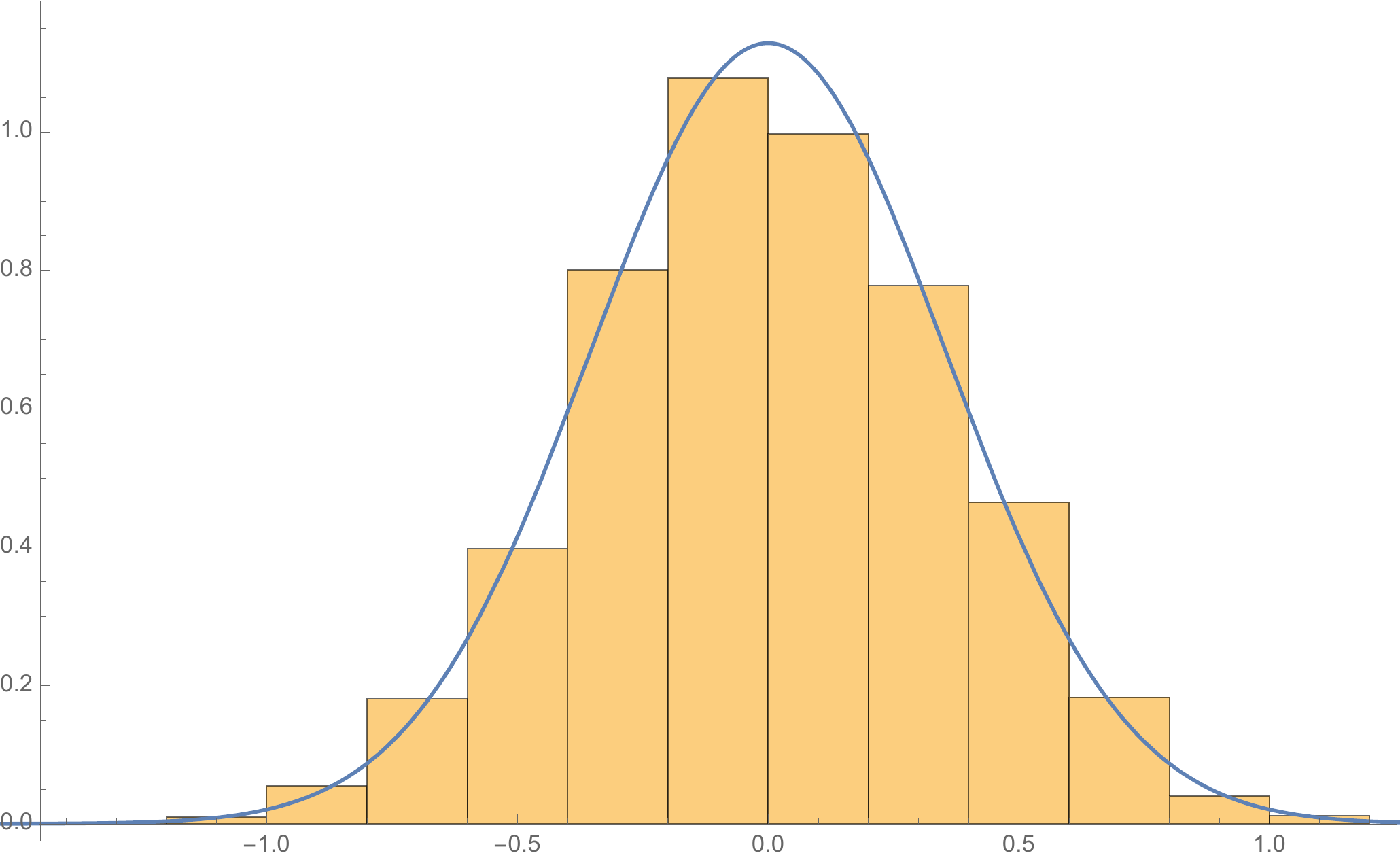}

\end{minipage}
\caption{The figure on the right shows one instance of a random walk of length $N=5000$ and $Y_0=0$ (solid line), along with the expected path from equation (\ref{eq:smallMoments}) (dashed line). The figure on the left shows a histogram of fluctuations away from the mean at $X=2500$ (for walk of length $N=5000$ and $Y_0=0$) scaled by $\frac{1}{\sqrt{N}}$, sampled from 5000 instances of the random walk. The overlaid curve is a normal distribution with mean 0 and variance $2x^2(1-x)^2$ with $x=\frac{1}{2}$. }
\label{fig:smallMoments}
\end{figure}
Figure \ref{fig:smallMoments} show a comparison of the derived formula with numeric simulations of the random walk.

\begin{remark}
It appears that in eqn. (\ref{eq:smallMoments}) for a walk of length, say, $N=3$ that the second moment should go to infinity due to the factor of $N-3$ in the denominator. However, one can check that for any $X\in\{0,1,2,3\}$ there is a corresponding term in the numerator that also goes to zero. Cancelling these gives the correct finite value for the moment. As we will only consider walks where $N$ is very large, we will ignore this subtlety.
\end{remark}

It is also possible to calculate the mixed moments recursively. 
\begin{prop}
Fix nonnegative integers $a_1,a_2,\ldots,a_n$, and consider points $X_1<X_2<\ldots<X_n$. Define $F_{\vec{a}}(X_1,...,X_n)=H(X_1)^{a_1}  \ldots H(X_n)^{a_n}$. Then this mixed moments can be recursively calculated by

{\small
\begin{align*}
\mathbb{E}[F_{\vec{a}}(X_1,...,X_n)] & =  \mathbb{E}[F_{\vec{a}}(X_1,...,X_{n-1})] \\
& + \sum_{k=1}^{a_n} \left(\binom{a_n}{k} -(1+ (-1)^{a_n-k}) \binom{a_n}{k-1} \frac{1}{N-X_n}\right) \mathbb{E}[F_{\vec{a}}(X_1,...,X_{n-1})H(X_n-1)^k]. 
\end{align*}
}

\end{prop}
\begin{proof}
Fix $X_1<X_2<\ldots<X_n$ and $a_1,\ldots,a_{n} \in \mathbb{N}$.  From the law of total expectation we have
\begin{align*}
\mathbb{E}[F_{\vec{a}}(X_1,...,X_n)] & = \mathbb{E}_Y[\mathbb{E}[F_{\vec{a}}(X_1,...,X_n) | H(X_1)=Y_1, \ldots, H(X_{n-1})=Y_{n-1}]] \\
& =  \mathbb{E}_Y[Y_1^{a_1}\ldots Y_{n-1}^{a_{n-1}}\mathbb{E}[H(X_n)^{a_n}|H(X_{n-1})=Y_{n-1}]]
\end{align*}
where we use the Markov property to simplify the conditional expectation. Note that after conditioning $H(X_{n-1})=Y_{n-1}$, the rest of the walk behaves as the original walk with shifted coordinate $X\to X-X_{n-1}$ and $N\to N-X_{n-1}$. This implies that
{\tiny
\[
\mathbb{E}[H(X_n)^{a_n}|H(X_{n-1})=Y_{n-1}] =  1+ \sum_{k=1}^{a_n} \left(\binom{a_n}{k} -(1+ (-1)^{a_n-k}) \binom{a_n}{k-1} \frac{1}{N-X_n+1}\right) \mathbb{E}[H(X_n-1)^k|H(X_{n-1})=Y_{n-1}]
\]
}
using Proposition (\ref{prop:recMoments}). Plugging this into the above gives the desired result.
\end{proof}

A similar calculation to the corollary of Proposition (\ref{prop:recMoments}) proves the following.
\begin{cor}
Fix nonnegative integers $a_1,\ldots, a_n$ and points $X_1<\ldots<X_n$. Then
\[
F_{\vec{a}}(X_1,\ldots,X_n)  = (N-X_n)P_{\vec{a}}(X_1,\ldots,X_n,Y_0)
\]
where $P_{\vec{a}}(x_1,\ldots,x_n,y)$ is a polynomial in $1+\sum_{i=1}^n a_i$ variables.
\end{cor}

As a useful example, we can write down a formula for the covariance. Fix $X_1<X_2$. We have
{\tiny
\begin{equation} \label{eq:covariance}
\begin{aligned}
Cov[H(X_1),H(X_2)] = & \mathbb{E}[H(X_1)H(X_2)] - \mathbb{E}[H(X_1)] \mathbb{E}[H(X_2)] \\
= &  \frac{2X_1(X_1-1)(N-X_2)(N-X_2-1)}{(N-1)^2(N-3)} \\
& + Y_0 \frac{2X_1(2N^2-5N+1-(3N-5)X_1)(N-X_2)(N-X_2-1)}{N(N-1)^2(N-2)(N-3)} \\
& - Y_0^2 \frac{2X_1(2N ^2-6N+3-(2N-3)X_1)(N-X_2)(N-X_2-1)}{N^2(N-1)^2(N-2)^2(N-3)}.
\end{aligned}
\end{equation}
}
Figure \ref{fig:Covariance} shows this covariance in the case $Y_0=0$.
\begin{figure}
\includegraphics[width=0.5\linewidth]{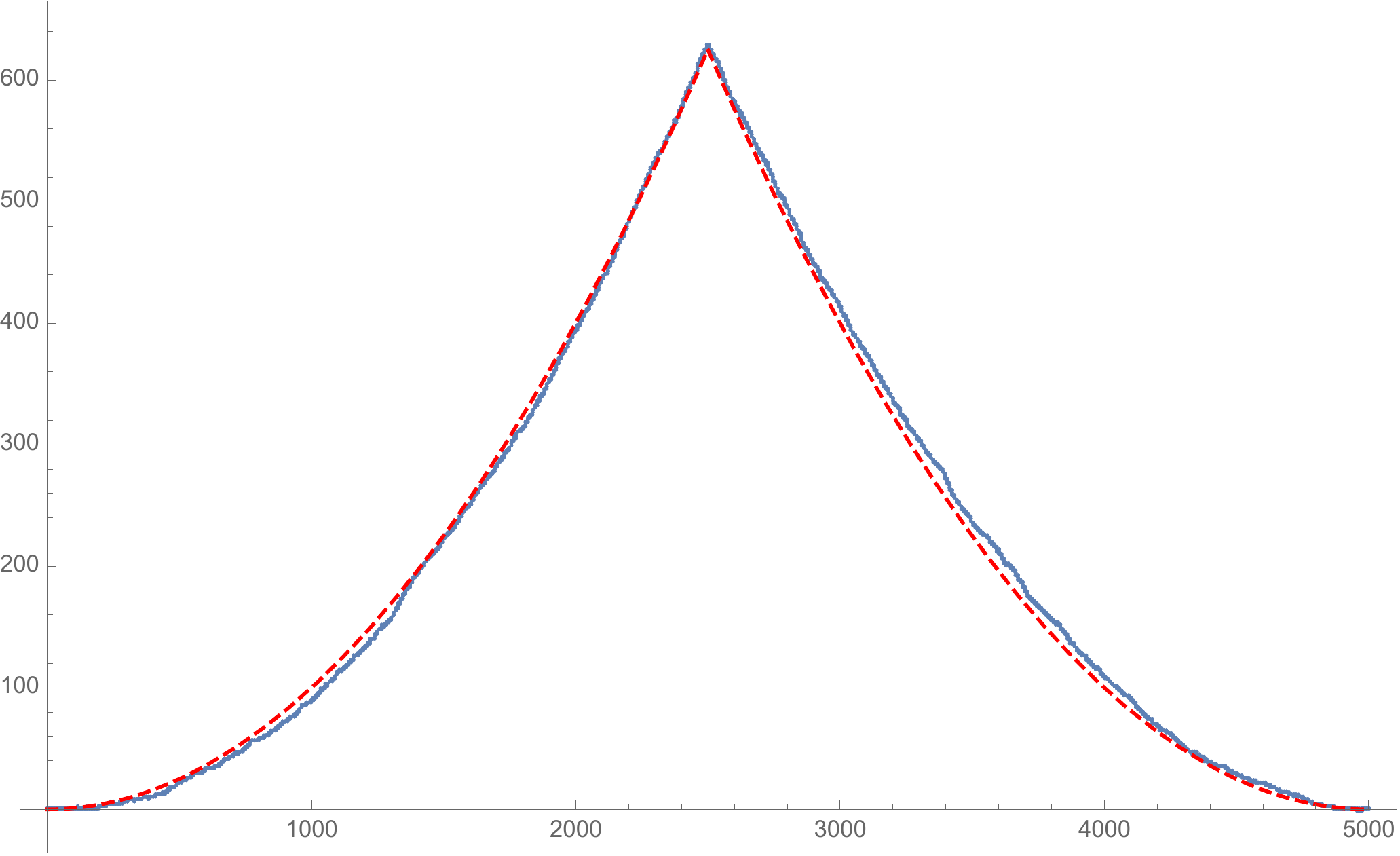}
\caption{The solid curve sample covariance of a walk of length $N=5000$ and $Y_0=0$ with one point fixed at $X=2500$, averaged over 5000 walks. The dashed curve is the curve calculated in equation (\ref{eq:covariance}).}
\label{fig:Covariance}
\end{figure}

Using equations (\ref{eq:smallMoments}) and (\ref{eq:covariance}), its easy to prove the following result from \cite{HZ}.
\begin{cor}
The volume $V$ of an oscillating tableaux of length $N$ and shape $|\lambda|=Y_0$, has the properties
\begin{align*}
& \mathbb{E}[V] = \frac{N(N+1)}{6} + Y_0 \frac{N+1}{3} \\
& Var[V] =  \frac{(N+1)N(N-2)}{45} + Y_0 \frac{(N+1)(3N+2)}{45} - Y_0^2\frac{4(N+1)}{45}
\end{align*}
\end{cor}
\begin{proof}
Let $Y_i = H(X_i)$. Write the volume as $V = \sum_{i=0}^N Y_i$. Then we have
\begin{align*}
\mathbb{E}[V] & =  \sum_{i=0}^N \mathbb{E}[Y_i] \\
& =  \sum_{i=0}^N \left( \frac{i(N-i)}{N-1} + Y_0 \frac{(N-i)(N-i-1)}{N(N-1)} \right) \\
& =  \frac{N(N+1)}{6} + Y_0 \frac{N+1}{3}
\end{align*}
as in \cite{HZ}. Further, we can calculate the variance
\begin{align*}
Var[V] & = Var\left[\sum_{i=0}^N Y_i\right] \\
& = \sum_{i,j=0}^N Cov[Y_i,Y_j] \\
& = \frac{(N+1)N(N-2)}{45} + Y_0 \frac{(N+1)(3N+2)}{45} - Y_0^2\frac{4(N+1)}{45}
\end{align*}
as desired.
\end{proof}
Figure \ref{fig:Volume} compare the above equations to the numerically sampled volume when $N$ is large.
\begin{figure}
\includegraphics[width=0.5\linewidth]{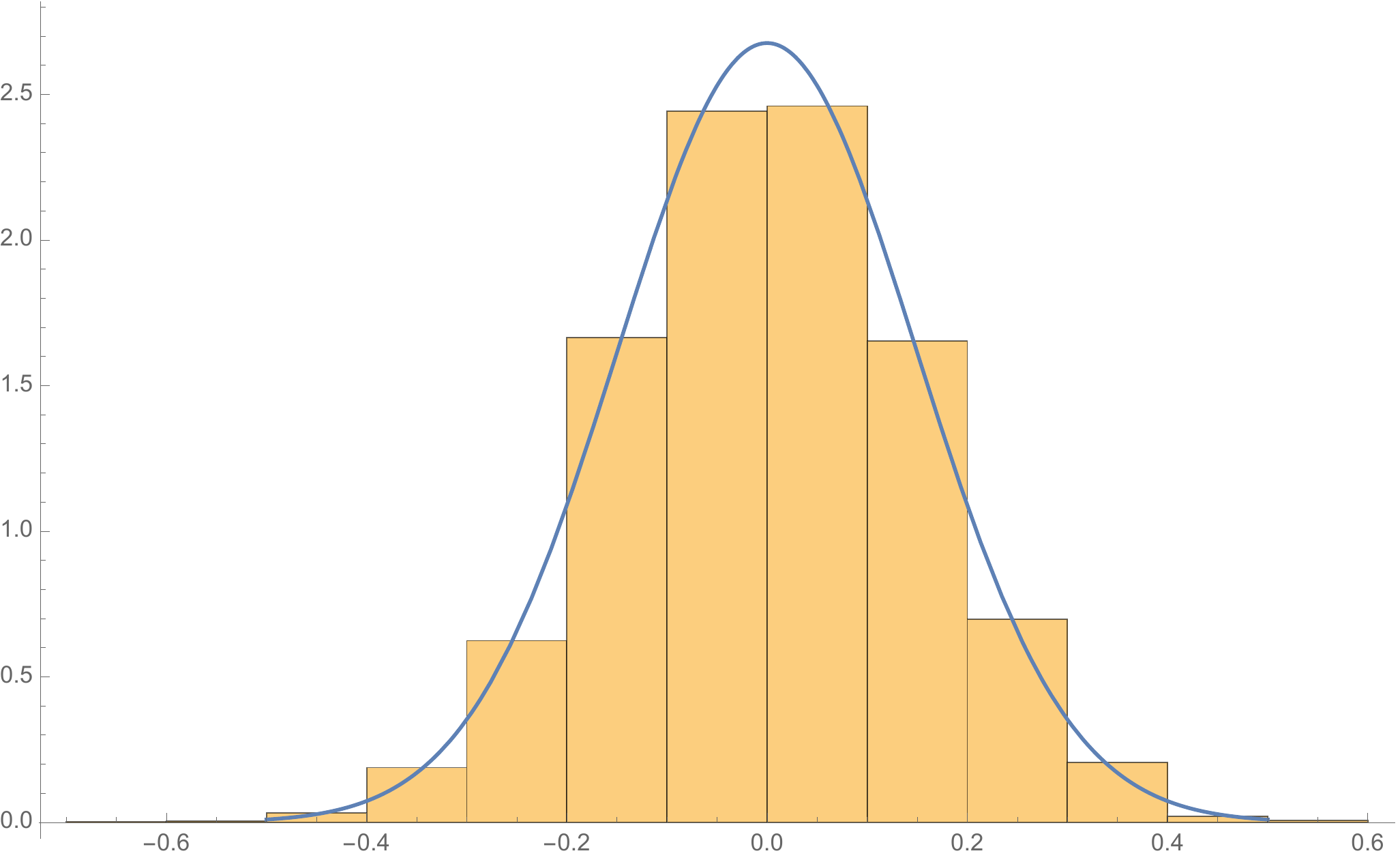}
\caption{A histogram of 5000 samples of the quantity $\frac{V-N^2/6}{\sqrt{N^3}}$, where $V$ is the volume of of an oscillating tableaux of length $N=5000$ from the empty partition to the empty partition. The overlaid curve is a normal distribution with mean 0 and variance $\frac{1}{45}$.}
\label{fig:Volume}
\end{figure}

Note that in \cite{HX} (Theorem 1.3), the authors extended the above by showing the polynomality of certain weighted averages of oscillating tableaux. It is easy to rephrase there results in this setting.
\begin{thm}[From \cite{HX}]
Let $P(x,y)$ be a polynomial of two variables. Then there exists a polynomial $Q(x,y)$ with the same degree and constant term as $P(x,y)$ such that
\begin{equation}
\mathbb{E}\left[\sum_{i=0}^N P(i,Y_i) \right] = (N+1) Q(N,Y_0)
\end{equation}
where $Y_i=H(i)$.
\end{thm}

\begin{remark}
The fact that we can recursively calculate the moments as in Proposition (\ref{prop:recMoments}), in spite of the position-dependent weights, is somewhat special. For example, the weights $p(x,y) = 1- \left(\frac{Y}{N-X}\right)^k$, $q(x,y) = \left(\frac{Y}{N-X}\right)^k$, $k>1$, still share the property that the $\mathbb{P}[H(N)=0]=1$. But the analogue of Proposition (\ref{prop:recMoments}) for the first moment becomes
\[
\mathbb{E}[H(X+1)]=1+\mathbb{E}[H(X)] -\frac{2\mathbb{E}[H(X)^k]}{(N-X)^k}
\]
which already depends on higher moments of the walk.
\end{remark}

\section{Master Equation}  \label{sec:3}

Define $p(X,Y)$ by $p(X,Y) = \mathbb{P}[H(X)=Y]$. From the Markov property of the random walk, we know that this satisfies the relation
{\small
\begin{equation}\label{eq:masterDiscrete}
\begin{aligned}
p(X+1,Y)  = & \mathbb{P}[H(X+1)=Y | H(X) = Y-1] p(X,Y-1) \\
& \hspace{1.5cm} +  \mathbb{P}[H(X+1)=Y | H(X) = Y+1] p(X,Y+1) \\
 = & \left( 1 - \frac{Y-1}{N-X} \right) p(X,Y-1) + \frac{Y+1}{N-X} p(X,Y+1)
\end{aligned}
\end{equation}
}
with initial condition $p(0,Y_0)=1$. This is the Master equation for the distribution.

We are interested in solutions to eqn. (\ref{eq:masterDiscrete}) in the limit $X=Nx$, $Y=Ny$, $Y_0=Ny_0$, and $N\to \infty$ (the Fokker-Planck equation for the distribution). Note in this limit $0<x<1$ and $0<y_0<1$. We'll write continuum trajectory of the walk as $h(x)=y$ where $h(x) = \frac{1}{N}H(Nx)$ with $N\to\infty$. We abuse notation and let $p(x,y)$ be the continuum distribution as well. 

Expanding in powers of $\frac{1}{N}$, we obtain
{\tiny
\begin{align*}
p(x,y) + \frac{1}{N} \partial_x p(x,y) + \frac{1}{2N^2} \partial_x^2 p(x,y) = & \left(1 - \frac{y}{1-x} + \frac{1}{N}\frac{1}{1-x} \right) \left(p(x,y) - \frac{1}{N}\partial_y p(x,y) + \frac{1}{2N^2} \partial_y^2 p(x,y)\right) \\
& +\left(\frac{y}{1-x} + \frac{1}{N} \frac{1}{1-x} \right)\left(p(x,y) + \frac{1}{N}\partial_y p(x,y) + \frac{1}{2N^2} \partial_y^2 p(x,y)\right) \\
& + O\left(\frac{1}{N^3}\right)
\end{align*}
}
which simplifies to 
\begin{equation} \label{eq:masterContinuum}
\begin{aligned}
 \frac{1}{N} \partial_x p(x,y) + \frac{1}{2N^2} \partial_x^2 p(x,y) = & -\frac{1}{N}\partial_y\left( \left(1- \frac{2y}{1-x}\right)p(x,y) \right) \\
 & + \frac{1}{2N^2} \partial_y^2 p(x,y) + O\left(\frac{1}{N^3}\right).
 \end{aligned}
\end{equation}

Let's look at the first order terms. We'll prove the following
\begin{prop} \label{prop:pde1}
The solution to 
\[
 \partial_x p + \partial_y\left( \left(1- \frac{2y}{1-x}\right)p(x,y) \right) = 0
\]
with initial condition $p(0,y)=\delta(y-y_0)$ is given by
\[
p(x,y) = \delta\left(y-x(1-x) -y_0(1-x)^2\right).
\]
\end{prop}
\begin{proof}
We can solve this by characteristics. First note that we can rewrite the pde as 
\[
p = \frac{1-x}{2} \partial_x p + \frac{1-x-2y}{2}\partial_y p.
\]
Suppose there is a curve $x=x(r),y=y(r)$ with $x(0)=0, y(0)=y_0$, such that $\frac{dp}{dr} = p$. We have
\[
p = \partial_x p \frac{dx}{dr} + \partial_y p \frac{dy}{dr}
\]
which implies 
\[
\begin{aligned}
& \frac{dx}{dr} =  \frac{1-x}{2} \\
& \frac{dy}{dr} = \frac{1-x-2y}{2}.
\end{aligned}
\]
This pair of equations can be solved to give
\[
\begin{aligned}
& x(r) = 1-e^{-\frac{1}{2}r} \\
& y(r) = e^{-\frac{1}{2}r} - e^{-r} + y_0 e^{-r} = x(1-x)+y_0(1-x)^2.
\end{aligned}
\]
Note a bit of algebra shows $r=-2\log\left(1-x\right)$. Along this characteristic curve we see that
\[
\frac{dp}{dr} = p \implies p(r) = p_0 e^r = p_0 \frac{1}{(1-x)^2}
\]
where $p_0$ is constant along the characteristics. Since $x(0)=0$ this means that $p_0$ can only depend on $y_0$, that is, $p_0 = F(y_0)=F\left(\frac{y-x(1-x)}{(1-x)^2}\right)$, for an arbitrary function $F$. Finally, using the initial condition $p(0,y)=\delta(y-y_0)$, we see that $F(y) =\delta(y-y_0)$ and thus
\[
\begin{aligned}
p(x,y) = & \frac{1}{(1-x)^2}F\left(\frac{y-x(1-x)}{(1-x)^2}\right) \\
= & \frac{1}{(1-x)^2} \delta\left(\frac{y-x(1-x)}{(1-x)^2} - y_0\right) \\
= & \delta\left(y-x(1-x)-y_0(1-x)^2\right).
\end{aligned}
\]
\end{proof}
So we see that for large $N$ the distribution of the random walk, with $p(0,y)=\delta(y-y_0)$ initial condition, concentrates along the curve $h(x)=x(1-x)+y_0(1-x)^2$. To compare this with the results from the previous section, note that in the above limit equation (\ref{eq:smallMoments}) becomes
{\small
\begin{align*}
& \mathbb{E}\left[\frac{1}{N} H(Nx)\right] = x(1-x) + y_0 (1-x)^2 + O\left(\frac{1}{N}\right) \\
& Var\left[\frac{1}{N}H(Nx)\right] = \frac{1}{N} \left(2x^2(1-x)^2 +2y_0x(2-3x)(1-x)^2 -4y_0^2 x(1-x)^3 \right) +  O\left(\frac{1}{N^2}\right)
\end{align*}
}
in agreement with Proposition (\ref{prop:pde1}). 

We're now interested in the fluctuations of the random around this limiting curve. To simplify the calculations, we will consider only the case when $y_0=0$, although the general case can be done similarly. It is expected that fluctuations should appear at a length scale of $\frac{1}{\sqrt{N}}$. We will see that at this scale we do in fact find nontrivial behavior. To this end, let $\tilde y$ be the fluctuation away from the mean. Consider the change of coordinates
\[
y = x(1-x) +\frac{1}{\sqrt{N}}\tilde y \implies \tilde y = \sqrt{N} (y-x(1-x)).
\]
We want to describe the distribution of the fluctuations $p(x,\tilde y)$. The following proposition does the job.
\begin{prop}\label{prop:pde2}
The distribution of the fluctuations satisfies
\[
\partial_x p = \partial_{\tilde y} \left( \frac{2\tilde y}{1-x} p\right) + 2x(1-x)\partial_{\tilde y}^2p
\] 
with initial condition $p(0,\tilde y) = \delta(\tilde y)$. This is solved by
\[
p(x,\tilde y) = \frac{1}{\sqrt{2\pi \sigma^2(x)}} e^{-\frac{\tilde y^2}{2\sigma^2(x)}}
\]
where $\sigma^2(x) = 2x^2(1-x)^2$.
\end{prop}
\begin{proof}
First note that under the change of coordinates above we have
\[
\begin{aligned}
& \partial_x \to \partial_x - \sqrt{N} (1-2x) \partial_{\tilde y} \\
& \partial^2_x \to \partial^2_x - 2\sqrt{N} (1-2x) \partial_{x\tilde y} + 4 \sqrt{N} \partial_{\tilde y} + N (1-2x)^2 \partial^2_{\tilde y}\\
& \partial_y \to \sqrt{N}\partial_{\tilde y} \\
& \partial^2_y \to N\partial^2_{\tilde y}
\end{aligned}
\]
as well as $1-\frac{2y}{1-x} = 1- 2x - \frac{1}{\sqrt{N}}\frac{2\tilde y}{1-x}$. Under these changes equation (\ref{eq:masterContinuum}) becomes
\[
\begin{aligned}
& \frac{1}{N}\partial_x p - \frac{1}{\sqrt{N}} (1-2x) \partial_{\tilde y}p + \frac{1}{2N} (1-2x)^2 \partial^2_{\tilde y}p \\
= &-\frac{1}{\sqrt{N}} \partial_{\tilde y} \left(\left(1-2x-\frac{1}{\sqrt{N}}\frac{2\tilde y}{1-x}\right)p \right) +\frac{1}{2N}\partial^2_{\tilde y}p + O\left( \frac{1}{N^\frac{3}{2}}\right)
\end{aligned}
\]
keeping only terms up to order $\frac{1}{N}$. This simplifies to
\[
\partial_x p =  \partial_{\tilde y} \left( \frac{2\tilde y}{1-x} p\right) + 2x(1-x)\partial_{\tilde y}^2p
\]
as desired. One can easily verify that
\[
p(x,\tilde y) = \frac{1}{\sqrt{2\pi \sigma^2(x)}} e^{-\frac{\tilde y^2}{2\sigma^2(x)}},
\]
with $\sigma^2(x) = 2x^2(1-x)^2$, solves this differential equation with the given initial condition.
\end{proof}
An identical calculation when $y_0\ne 0$ results in a solution of the same form with 
\begin{equation}\label{eq:varVar}
\sigma^2(x) = 2x(1-x)^2 + 2y_0x(2-3x)(1-x)^2 - 4y_0^2x(1-x)^3
\end{equation}
Note that the solution above is in fact a probability distribution as 
\[
\int_0^1\int_{-\infty}^{\infty} \; p(x,\tilde y) \; d\tilde y dx =1.
\]
 Note as well that for any fixed $x$, the marginal distribution of the fluctuation at $x$ is Gaussian with variance $2x^2(1-x)^2$. In particular, the variance goes to zero at $x=1$ as expected.

Continuing, we can ask about the joint distribution of the fluctuations at different points $x_1,\ldots,x_n$.
\begin{prop}
For any finite number of points $x_1<x_2<\ldots<x_n$, the joint distribution of the fluctuations are multivariate Gaussian with covariance matrix 
\[
\mathbf{C} = \left[2\min(x_i,x_j)^2(1-\max(x_i,x_j))^2 \right]_{i,j=1}^n.
\]
\end{prop}
To prove this we first need some facts about the matrix $C$ given above.
\begin{lem}\label{lem:covarianceMatrix}
For $x_n<x_m$, define $z_{n,m} = 2(x_m-x_n)(x_m-2x_mx_n+x_n)$ and $c_{n,m} = 2x_n^2(1-x_m)^2$. Let the covariance matrix for $n$ points $x_1<\ldots<x_n$ be $\mathbf{C}_{(n)}$.  It has the properties
\[
\begin{aligned}
& det\left(\mathbf{C}_{(n)}\right) = c_{1,n} \prod_{i=2}^{n} z_{i-1,i} \\
& \mathbf{C}_{(n)}^{-1} = \mathbf{A} + \mathbf{B}, \hspace{1cm} \mathbf{C}_{(1)}^{-1} = \frac{1}{2x_1^2(1-x_1)^2}
\end{aligned}
\]
where $\mathbf{A}$ has $\mathbf{C}_{(n-1)}^{-1}$ in the upper-left corner and zeros elsewhere, $\mathbf{B}$ has a nonzero $2\times 2$ block  $\mathbf{M} =\begin{pmatrix}\frac{(1-x_n)^2}{(1-x_{n-1})^2 z_{n-1,n}} & -\frac{1}{z_{n-1,n}} \\-\frac{1}{z_{n-1,n}} & \frac{(1-x_{n-1})^2}{(1-x_{n})^2 z_{n-1,n}}  \end{pmatrix}$ in the lower-right corner and zeros elsewhere.
\end{lem}
\begin{proof}
It is easy to see the lemma implies $\mathbf{C}_{(n)}^{-1}$ is a tridiagonal matrix. One can compute explicitly the form of this matrix, then check it is in fact the inverse.

Now we'll take the determinant of the inverse. Applying the sequence of row operations $r_2\to r_2+\frac{x_1^2}{x_2^2}r_1,\ldots, r_n\to r_n+\frac{x_{n-1}^2}{x_n^2}r_{n-1}$ and simplifying results in an upper triangular matrix with diagonal $(\frac{x_2^2}{x_1^2 z_{1,2}},\ldots, \frac{x_n^2}{x_{n-1}^2 z_{n-1,n}},\frac{1}{2x_n^2(1-x_n)^2})$, from which we get the desired formula.
\end{proof}

Now using this, we return to the proposition.
\begin{proof}[Proof of proposition:]
First let's suppose the proposition is true. Let $\vec{y}^T_{(n)}= (\tilde y_1, \ldots , \tilde y_n)$. If the proposition holds then
{\
\begin{equation} \label{eq:jointStep}
\begin{aligned}
\mathbb{P}[\vec{y}_{(n)}] = & \frac{1}{\sqrt{(2\pi)^n|\mathbf{C}_{(n)}|}} e^{-\frac{1}{2}\vec{y}_{(n)}^T\mathbf{C}_{(n)}^{-1}\vec{y}_{(n)}}  \\
= & \frac{1}{\sqrt{(2\pi)^{n-1}|\mathbf{C}_{(n-1)}|}} \frac{1}{\sqrt{2\pi c_{1,n}c_{1,n-1}^{-1}z_{n-1,n}}}  e^{-\frac{1}{2}\vec{y}_{(n-1)}^T\mathbf{C}_{(n-1)}^{-1}\vec{y}_{(n-1)}} \\
& \hspace{2cm} \times exp\left(-\frac{1}{2} \begin{pmatrix} \tilde y_{n-1} &  \tilde y_{n} \end{pmatrix}  \mathbf{M}   \begin{pmatrix} \tilde y_{n-1} \\  \tilde y_{n} \end{pmatrix}\right) 
\end{aligned}
\end{equation}
}
where we use the lemma (as well as the notation) from above.

Now we prove the proposition by induction. We know this holds for a single point. Now assume it holds for $n-1$ points. We have
{
\[
\begin{aligned}
\mathbb{P}[\vec{y}_{(n)}] & = \mathbb{P}[\vec{y}_{(n-1)}] \mathbb{P}[\tilde y_n |  \tilde y_{n-1}]  \\
& = \frac{e^{-\frac{1}{2}\vec{y}^T_{(n-1)}\mathbf{C}^{-1}_{(n-1)}\vec{y}_{(n-1)}}}{\sqrt{(2\pi)^{n-1} |\mathbf{C}_{(n-1)}|}}  \frac{1}{\sqrt{2\pi \sigma^2(x_n)}} e^{-\frac{N\left(y_n -\frac{(x_n-x_{n-1})(1-x_n)}{1-x_{n-1}} - y_{n-1}\frac{(1-x_n)^2}{(1-x_{n-1})^2}\right)^2}{2\sigma^2(x_n)}}
\end{aligned}
\]
}
where the second factor comes from considering a walk of (continuum) length $1-x_{n-1}$ beginning at height $y_{n-1} = x_{n-1}(1-x_{n-1}) + \frac{1}{\sqrt{N}}\tilde y_{n-1}$. In this case,
{\small
\[
\begin{aligned}
\sigma^2(x_n) =  \frac{2(x_n-x_{n-1})(1-x_n)^2}{(1-x_{n-1})^3} & + y_{n-1} \frac{2(x_n-x_{n-1})(1-x_n)^2(2+x_{n-1}-3x_n)}{(1-x_{n-1})^4} \\
& - y_{n-1}^2\frac{4(x_n-x_{n-1})(1-x_n)^3}{(1-x_{n-1})^5}.
\end{aligned}
\]
}
This can be greatly simplified by replacing $y_{n-1}$ in the variance by $x_{n-1}(1-x_{n-1})$ since any corrections are of sub-leading order. We then have
\[
\sigma^2(x_n) = \frac{2(x_n-x_{n-1})(1-x_n)^2(x_n-2x_nx_{n-1}+x_{n-1})}{(1-x_{n-1})^2} = \frac{c_{1,n}}{c_{1,n-1}}z_{n-1,n}
\]
plus terms of size $O\left(\frac{1}{\sqrt{N}}\right)$. Comparing with equation (\ref{eq:jointStep}), it is left to show that
\[
-\frac{N\left(y_n - \frac{(x_n-x_{n-1})(1-x_n)}{1-x_{n-1}} - y_{n-1}\frac{(1-x_n)^2}{(1-x_{n-1})^2}\right)^2}{2\sigma^2(x_n)} = -\frac{1}{2} \begin{pmatrix} \tilde y_{n-1} &  \tilde y_{n} \end{pmatrix}  \mathbf{M}   \begin{pmatrix} \tilde y_{n-1} \\  \tilde y_{n} \end{pmatrix}
\]
but this is a simple calculation.
\end{proof}

Together the above two propositions, \ref{prop:pde1} and \ref{prop:pde2}, prove the following:
\begin{thm}
The fluctuations of this random walk around its mean $h_0(x)=x(1-x)$ form a Gaussian process with covariance function
\[
C(x_i,x_j) =2\min(x_i,x_j)^2(1-\max(x_i,x_j))^2.
\]
\end{thm}

\begin{cor}
As $N\to \infty$, the (appropriately scaled) volume  of an oscillating tableaux from the empty partition to the empty partition of length $N$ is Gaussian distributed with mean $\frac{1}{6}$ and variance $\frac{1}{45}$.
\end{cor}

\section{Variational Principle}  \label{sec:4}
The variational principle is a common tool in the study of random walk (for a recent example see \cite{DGR}). It is well-known that for a random walk from the origin to $(N,M)$ on $\mathbb{Z}^2$ consisting of elementary steps $(1,1)$  and $(1,-1)$ with \emph{constant} weight $p$ and $q$, respectively, that the number of paths $Z$ grows like
\[
Z \sim e^{N\sigma(v)}
\]
where $N$ is the length of the walk and $M/N \to v$ as $N$ becomes large. Here $\sigma$ is known as the surface tension and is given explicitly by
\begin{equation}\label{eq:surfaceTension}
\sigma(v) = \frac{1+v}{2} \log(p) + \frac{1-v}{2} \log(q) - \frac{1+v}{2} \log\left(\frac{1+v}{2}\right) - \frac{1-v}{2} \log\left(\frac{1-v}{2}\right).
\end{equation}
Using this one would consider the action on trajectories $\varphi$ given by
\begin{equation}
S[\varphi(x)] = \int \; \sigma(\partial_x \varphi) \; dx
\end{equation}
whose extremum $\varphi_0$ would correspond to the expected trajectory of the walk. Looking at the second order term in the expansion of the action around $\varphi_0$ gives an effective action describing the behavior of the fluctuations of the walk. 

Although surface tension given above is only valid for constant weights $p,q$, it is reasonable to expect that replacing $p$ and $q$ with position dependent weights $p(x,y),q(x,y)$ will only add correction at sub-leading order if the weights are sufficiently smooth. Regardless, we conjecture that the surface tension is still valid with $p(x,y) = 1-\frac{y}{1-x}$ and $q(x,y) = \frac{y}{1-x}$. 

Let $h(x)$ be a continuum trajectory of our random walk. Without loss of generality, we restrict ourselves to paths with $h(0)=h(1)=0$ (the first condition is the statement $y_0=0$, and the second is harmless as $\mathbb{P}[h(1)=0]=1$). As we will be repeatedly integrating by parts, this restriction allows us to discard the resulting boundary terms. Under these conditions, we have
\begin{prop}
The action is extremized by the trajectory $h(x)$ solving
\begin{equation}
2hh''-(h')^2+1=0
\end{equation}
with $h(0)=h(1)=0$.
\end{prop}
\begin{proof}
Consider the trajectory $h(x) + \epsilon h_1(x)$ where $\epsilon$ is some small parameter and $h_1(0)=h_1(1)=0$. Expanding the action in powers of $\epsilon$ a standard computation gives
{\tiny
\begin{align*}
S[& h(x) +  \epsilon h_1(x)]-S[h(x)] = \\
= &  \epsilon \int_0^1 \; \left(\left( \tan^{-1}\left(1-\frac{2 h(x)}{1-x}\right) - \tan^{-1}(h'(x)) \right) h_1'(x)  - \frac{2h(x)-(1-x)(1-h'(x))}{2h(x)(1-x-h(x))}h_1(x) \right) \; dx \\
+ & \frac{\epsilon^2}{2} \int_0^1 \; \left(\frac{1}{2}\left(\frac{1+h'(x)}{(1-x-h(x)^2)} -\frac{1-h'(x)}{h(x)^2}\right)h_1(x)^2-\frac{1-x}{h(x)(1-x-h(x))}h_1(x)h_1'(x) - \frac{1}{1-h'(x)^2}h_1'(x)^2\right) \; dx.
\end{align*} 
}
We want to find $h(x)$ such that the first-order term vanishes for all $h_1$. After an integration by parts, this condition gives the Euler-Lagrange equation
\[
\frac{d}{dx}\left( \tan^{-1}\left(1-\frac{2 h(x)}{1-x}\right) - \tan^{-1}(h'(x)) \right) + \frac{2h(x)-(1-x)(1-h'(x))}{2h(x)(1-x-h(x))} = 0
\]
which simplifies to 
\begin{equation}
2hh''-(h')^2+1=0.
\end{equation}
\end{proof}

Note that the positive solution to this ODE satisfying the boundary conditions $h(0)=h(1)=0$ is $h_0(x)=x(1-x)$, in agreement with our previous results. Expanding the action to second order about $h_0$ gives the following proposition.
\begin{prop}
The second variation of the action is given by
\begin{equation}
\delta S [h_1(x)] =-\frac{1}{2} \int_0^1 \;  h_1 \Delta h_1 \; dx
\end{equation}
where $\Delta$ is a massive Laplace-type differential operator,
\begin{equation}
\Delta = -\frac{1}{2x(1-x)} \frac{d^2}{dx^2} +\frac{1-2x}{2x^2(1-x)^2}\frac{d}{dx} + \frac{1}{x^2(1-x)^2}.
\end{equation}
\end{prop}
\begin{proof}
 We have already calculated the expansion of the action to second-order above. With $h(x)$ replace with $h_0(x)$, we'll call this term $ \delta S[h_1(x)]$. After some simplification this becomes
 \[
 \delta S[h_1(x)] = -\frac{1}{2} \int_0^1 \; \frac{(2h_1(x)+(1-x)h_1'(x))^2}{4x(1-x)^3} \; dx.
 \]
 To get it in the desired form, we expand the integrand as
 {\tiny
 \[
 \frac{(2h_1(x)+(1-x)h_1'(x))^2}{4x(1-x)^3} = h_1(x) \frac{2h_1(x)+(1-x)h_1'(x)}{2x(1-x)^3} + h_1'(x)\frac{2(1-x)h_1(x)+(1-x)^2h_1'(x)}{4x(1-x)^3},
 \]
 }
then integrating by parts on the second term, we can move the derivative from the leading $h_1'(x)$. Doing so and simplifying leaves us with
\[
h_1(x) \left(-\frac{1}{2x(1-x)} h_1''(x) +\frac{1-2x}{2x^2(1-x)^2}h_1'(x)+ \frac{1}{x^2(1-x)^2}h_1(x)\right)
\]
 as desired.
 
\end{proof}

One can easily check that the Greens function for this operator $\Delta$ is given by
\[
G(x_1,x_2) = 2x_1^2(1-x_2)^2
\]
where $x_1\le x_2$. Note this is precisely the covariance we calculated earlier!

\section{Further Study}

Here we found that the area of the partitions of an oscillating tableaux of length $N$ and shape $\lambda$ could be modeled as random walk starting at height $|\lambda|$ with certain position-dependent weights. These weights came from viewing the oscillating tableaux as a sequence of competing hook walks. It would be interesting to study other sets of weights interpolating between the these basic hook walks. For example, consider the weights $p_{remove} = \frac{1-q^{Y}}{1-q^{N-X}}$ and $p_{add} = 1 - \frac{1-q^{Y}}{1-q^{N-X}}$. As $q\to 1$ we recover the oscillating tableaux, and as $q\to0$ we have that $q_{remove}\to 1$. Note that the $q\to 0$ limit is related to the hook-walk method for calculation the number of Standard Young Tableaux \cite{GNW1,GNW2}. See figure \ref{fig:qWeights}.

\begin{figure}
\includegraphics[width=0.5\linewidth]{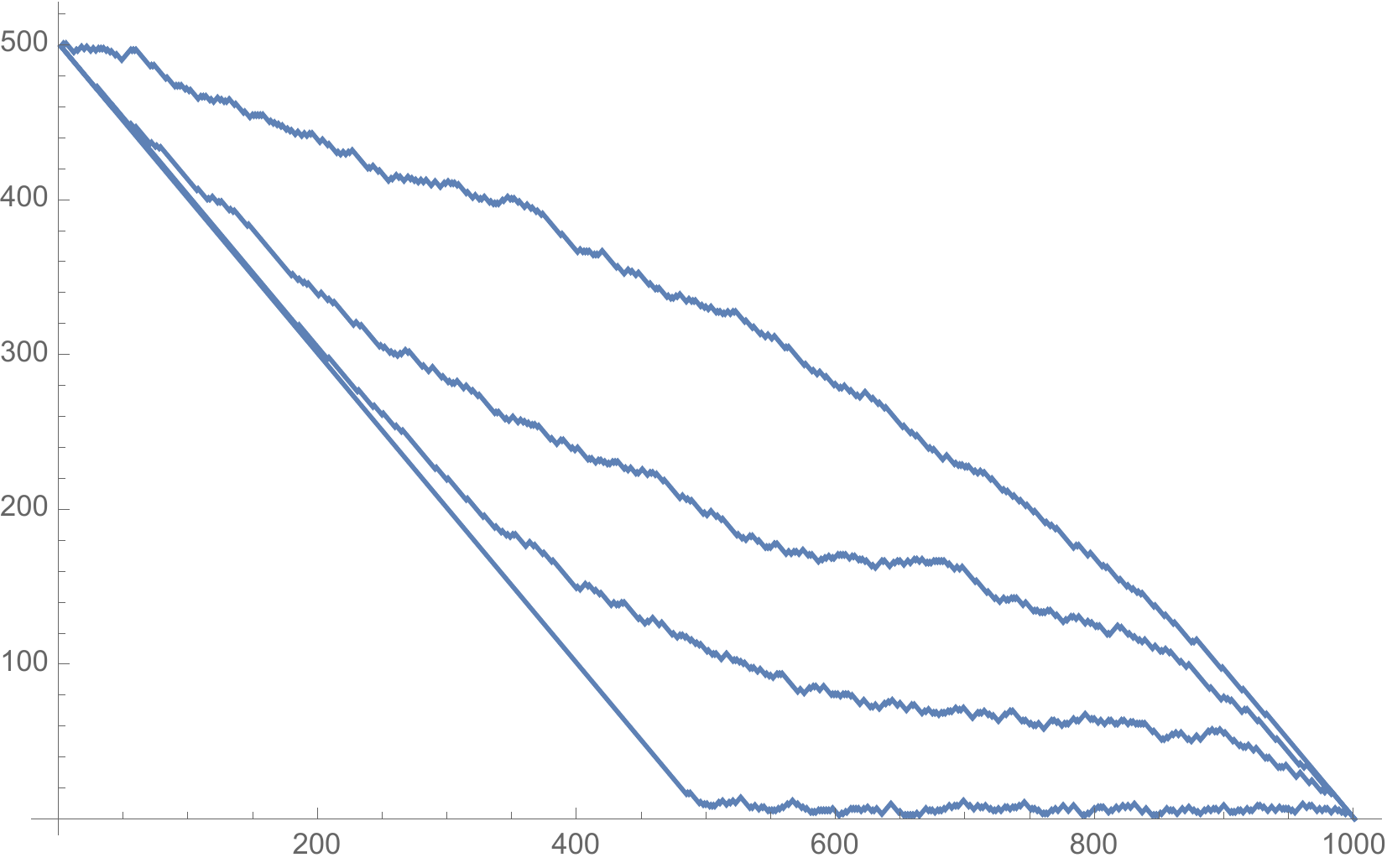}
\caption{Random walk of length $1000$ and $Y_0=500$ with weights given by $p_{remove} = \frac{1-q^{Y}}{1-q^{N-X}}$ and $p_{add} = 1 - \frac{1-q^{Y}}{1-q^{N-X}}$ for a variety of $q$. From ``highest" to ``lowest": $q = 0.999, 0.995, 0.99, 0.9$.}
\label{fig:qWeights}
\end{figure}
 
 In \cite{BO} the authors study Markov processes on the set of Young diagrams with more general jump rates than those discussed here. They show that the correlation functions between the length of rows of the partitions at different times forms a determinantal process, compute its kernel, and study its asymptotic behavior. It would be interesting to check of the the probabilities $p_{add}$, $p_{remove}$ given here fit into the framework given in \cite{BO}.

\end{document}